\documentclass[12pt]{amsart}
\usepackage[top=30truemm,bottom=30truemm,left=30truemm,right=30truemm]{geometry}
\UseRawInputEncoding
\usepackage{mathrsfs}
\usepackage{bm}
\usepackage{amsfonts,amssymb}
\usepackage{dsfont}
\usepackage{extarrows}
\usepackage{amsmath}
\usepackage{mathrsfs}
\usepackage{enumerate}
\usepackage{amscd}
\usepackage[all,2cell]{xy}
\usepackage{hyperref}

\theoremstyle{plain}

\theoremstyle{definition} 

\newtheorem{thm}{Theorem}[section]

\newtheorem{lem}[thm]{Lemma}

\theoremstyle{definition}
\newtheorem{defn}{Definition}[section]

\theoremstyle{remark}
\newtheorem{rem}{\bf Remark}[section]

\newcommand{\be}{\begin{equation}}
\newcommand{\ee}{\end{equation}}
\newcommand{\bea}{\begin{eqnarray}}
\newcommand{\eea}{\end{eqnarray}}
\newcommand{\ben}{\begin{eqnarray*}}
	\newcommand{\een}{\end{eqnarray*}}
\newcommand{\bt}{\begin{split}}
	\newcommand{\et}{\end{split}}
\newcommand{\bet}{\begin{equation}}
\newcommand{\mc}{\mathbb{C}}
\newcommand{\mr}{\mathbb{R}}
\newcommand{\ra}{\rightarrow}
\newcommand{\beq}{\begin{equation*}}
\newcommand{\eeq}{\end{equation*}}
\newcommand{\bal}{\begin{aligned}}
\newcommand{\eal}{\end{aligned}}

\newcommand{\dbar}{\bar{\partial}}
\newcommand{\pa}{\partial}
\newcommand{\bp}{\bar{\partial}}

\newcommand{\calO}{{\mathcal{O}}}%
\newcommand{\Ker}{{\textup{Ker}}}

\newcommand{\Dom}{{\textup{Dom}}}

\renewcommand{\leq}{\leqslant}%
\renewcommand{\geq}{\geqslant}%
\newcommand{\inner}[1]{\langle#1\rangle}
\newcommand{\iinner}[1]{\langle\langle#1\rangle\rangle}

\makeatletter
\g@addto@macro{\endabstract}{\@setabstract}
\newcommand{\authorfootnotes}{\renewcommand\thefootnote{\@fnsymbol\c@footnote}}%
\makeatother

\newcommand{\Levi}{{\textup{Levi}}}
\newcommand{\Hess}{{\textup{Hess}}}

\begin{document}

\title{The union problem for domains with partial pseudoconvex boundaries}
\date{}

\author[J. Hu]{Jinjin Hu}
\address{Jinjin Hu: Yau Mathematical Sciences Center, Tsinghua University, Beijing 100084, China.}
\email{hujinjin@mail.tsinghua.edu.cn}

\author[X.Zhang]{Xujun Zhang}
\address{Xujun Zhang: Guangxi Center for Mathematical Research, Guangxi University, Nanning, Guangxi, 530004, P. R. China}
\address{Guangxi Base, Tianyuan Mathematical Center in Southwest China,  Nanning, Guangxi, 530004, P. R. China}
\email{xujunzhang@amss.ac.cn}



\maketitle

\begin{abstract}
We show that a smooth bounded domain in $\mc^n$ admitting partial pseudoconvex exhaustion remains partial pseudoconvex.
The main ingredient of the proof is based on a new characterization of hyper-$q$-convex domains.
Furthermore,we get several convex analogies in $\mr^n$.    
\end{abstract}

\section{Introduction}
One of the classical problems in complex geometry is the following problem called the union problem: for $n\geq 2$,
if a complex manifold $X$ admits a sequence exhaustion $\{X_{k} \}^{\infty}_{k=1}$ such that $X=\bigcup_{n=1}^\infty X_n$, where $X_k \subset \subset X_{k+1}$,
the union problem asks whether $X$ can be described in terms of its exhausting manifolds.

The Behnke-Stein's Theorem shows that
the union of domains of holomorphy is a domain of holomorphy (\cite{BT1}). 
Due to the solution of Levi's problem,
one can prove that if a domain in $\mc^{n}$ admits a pseudoconvex exhaustion is still a pseudoconvex domain (\cite{Hor65}). 
Recently, inspired by the $L^{2}$ converse teichnique in \cite{DNW21},\cite{DNWZ22},\cite{Deng-Zhang},
Liu-Zhang showed that a domain with null thin complement $(\mathring{\overline{D}}=D)$ admits a complete K\"ahler exhaustion is still a complete K\"ahler domain (\cite{Liu-Zhang}).

In this note, we study the union problem for the domain with partial convex boundary. 
The first main result of this note is the following.
\begin{thm}\label{thm:main-1}
	Let $D$ be a bounded domain in $\mc^n$ with a smooth boundary. 
	Let $\{D_j\}$ be a sequence of open subsets of $D$ with $D_j\subset D_{j+1} $ and $\bigcup_jD_j=D$. 
If each $D_j$ is hyper-$q$-convex, 
	then $D$ is hyper-$q$-convex.
\end{thm}

The smooth domain $D$ is said to be hyper-$q$-convex if the sum of any \( q \) eigenvalues of the Levi form on the complex tangent space of $\partial D$ is non-negative.
Grauert and Riemenschneider introduced the related concept in \cite{Grauert-Riemenschneider}.
The $\dbar$-problem on hyper-$q$-convex domain in $\mc^n$ was studied in \cite{Ho-1991} and \cite{Ji-Tan-Yu-2015}.
We prove Theorem \ref{thm:main-1} by standard functional technique and the following characterization of the hyper-$q$-convex domain.

\begin{thm}\label{thm:main-2}
	Let $D$ be a bounded domain in $\mc^n$ with smooth boundary, $ 1\leq q \leq n-1$.
Suppose that for any strictly plurisubharmonic function $\varphi$ of the form: 
	$$
	\varphi=a\|z-z_0\|^2-b, \forall a, b > 0, \forall z_0 \in \mc^n,
	$$ 
	and any $\bar\partial$-closed form $f\in \wedge^{0,q}(\bar{D}) \cap \text{Dom}(\bar\partial^*_{\varphi})$,
	the equation $\bar\partial u=f$ is solvable on $D$ with the estimate
		$$
		\int_{D} |u|^2 e^{-\varphi} d\lambda    \leq \int_{D} \inner{\textup{Levi}_{\varphi}^{-1}f,f } e^{-\varphi }d\lambda    =\frac{1}{aq}\int_D|f|^2e^{-\varphi } d\lambda,
		$$
then $D$ is hyper-$q$-convex.
\end{thm}

We prove Theorem \ref{thm:main-2} by showing that the solvability of the $\dbar$ equation indicates a special version of Bochner-type inequality with boundary term.
We employ a choice of weight functions to derive a contradiction, following the localization technique used in the work of Deng, Ning, Wang, and Zhou (\cite{DNW21},\cite{DNWZ22},\cite{Deng-Zhang}).

\begin{rem}
	\begin{enumerate}
		\item It's well known that the solvablity of the $q$-th $\dbar$ equation indicates that $H^{q}(D,\calO)=0$.
		One can show that for a hyper-$q$-convex domain $D$,
		$H^{k}(D,\calO)=0$ for any $k \geq q$ (\cite{Hor65}, \cite{Ji-Tan-Yu-2015}).
		But due to Laufer's algorithm (\cite{Laufer}), 
		$H^{q}(D,\calO)=0$ does not imply $H^{k}(D,\calO)=0$ for any $k\geq q$,
therefore the estimate for the solution in Theorem \ref{thm:main-2} is essential.
				\item 	When $q=1$,
		the boundary regularity of Theorem \ref{thm:main-2} can be weakened to the null thin complement $(\mathring{\overline{D}}=D)$ due to the division theorem on domain $D$ by the recent work in \cite{Liu-Zhang}.
		However, the geometric interpretation of the division theorem for holomorphic $(0, q)$-forms remains unclear to the authors.
		Therefore it's unknown whether the boundary regularity of Theorem \ref{thm:main-2} can be weakened for $q \geq 2$.
	\end{enumerate}
\end{rem}

Similar to \cite{Deng-Zhang},
we get the convex analogy for Theorem \ref{thm:main-1} and Theorem \ref{thm:main-2}.

\begin{thm}\label{thm:main-3}
	Let $D$ be a bounded domain in $\mr^n$ with smooth boundary, $ 1\leq q \leq n-1$.
Suppose that for any strictly convex function $\varphi$ of the form: 
	$$
	\varphi=a\|x-x_0\|^2-b, \forall a, b > 0, \forall x_0 \in \mr^n,
	$$ 
	and any $d$-closed form $f\in \wedge^{q}(\bar{D}) \cap \text{Dom}(d^*_{\varphi})$,
	the equation $d u=f$ is solvable on $D$ with the estimate
		$$
		\int_{D} |u|^2 e^{-\varphi} d\lambda    \leq \int_{D} \inner{\Hess_{\varphi}^{-1}f,f } e^{-\varphi }  d\lambda  =\frac{1}{2aq}\int_D|f|^2e^{-\varphi } d\lambda ,
		$$
then $D$ is $q$-convex.
\end{thm}

\begin{thm}\label{thm:main-4}
	Let $D$ be a bounded domain in $\mr^n$ with smooth boundary. 
	Let $\{D_j\}$ be a sequence of open subsets of $D$ with $D_j\subset D_{j+1} $ and $\bigcup_jD_j=D$. If each  $D_j$ is $q$-convex, 
then $D$ is $q$-convex.
\end{thm}

A domain is called a \( q \)-convex domain (in the sense of Harvey-Lawson) if the sum of any \( q \) eigenvalues of the second fundamental form on the tangent space of $\partial D$ is non-negative  (\cite{Harvey-Lawson-12},\cite{Harvey-Lawson-13}). 
J.P.Sha and H.H.Wu proved that such domains are homotopy equivalent to a CW complex of dimension no greater than \( q-1 \) (\cite{WuHH-1987}, \cite{Sha-1986}).
When $q=1$, $q$-convex domain is just the convex domain in the usual sense.
$L^2$ existence results for $d$ operator on $q$-convex domain was studied in \cite{Ji-Liu-Yu-2014}.Theorem \ref{thm:main-4} can be seen as a compact version of Theorem 4.4 in \cite{Harvey-Lawson-12} with different approach,
the proof of Theorem \ref{thm:main-4} in our arguement is essentially similar to Theorem \ref{thm:main-2},
thus we omit the details of the proof of Theorem \ref{thm:main-4} in this note.

The rest of this note is organized as follows.
In \S \ref{sec:notations}, we clarify some notations.
We prove Theorem \ref{thm:main-2} in \S \ref{sec:partial-pseducocovex}, then we prove Theorem \ref{thm:main-1} in \S \ref{sec:proof-union}.
The Theorem \ref{thm:main-3} is proved in \S \ref{sec:convex-analogy}.

\section{Preliminary}\label{sec:notations}

\subsection{Notations for the $\dbar $ complex}
The coordinate on $\mc^n$ will be denoted by $z=(z_1,\cdots, z_n)$, with $z_j=x_j+iy_j$, $(j=1,\cdots, n).$
We assume that $D$ is a bounded domain in $\mc^n$ with smooth boundary.
Then there exits a smooth function $\rho:\mc^n \ra \mr$  such that
$$
D=\left\{z \in \mc^n; \rho(z)<0 \right\}
$$
and $\nabla \rho|_{\pa D} \neq 0$, where
$$\nabla\rho=\sum_{j=1}^n(\frac{\partial\rho}{\partial x_j}\frac{\partial}{\partial x_j}+\frac{\partial\rho}{\partial y_j}\frac{\partial}{\partial y_j})$$
is the gradient of $\rho$.
We take a normalization such that $|\nabla\rho|\equiv 1$ on $\partial D$.
Let $D$ be a domain in $\mc^n$.
We denote by $\wedge^{0,q}(D)$ the space of smooth $(0,q)$-forms on $D$, for any $0\leq q\leq n$ ((0,0)-forms are just smooth functions), and
$\wedge^{0,q}_c(D)$ the elements in
$\wedge^{0,q}(D)$ with compact support.
Let $\varphi$ be a real-valued continuous function on $D$.
Given $\alpha=\sum_I\alpha_Id\bar z_I, \beta=\sum_I\beta_Id\bar z_I\in \wedge^{0,q}(D)$, we define the products of $\alpha$ and $\beta$ and the corresponding norm with respect to $\varphi$ as follows:
$$
\begin{aligned}
	\inner{\alpha, \beta}_{\varphi} &= \sum_I \alpha_I \cdot \overline\beta_I e^{-\varphi},
	|\alpha|^2_{\varphi}=\inner{\alpha,\alpha}_{\varphi}. \\
	\iinner{\alpha, \beta}_{\varphi} &=\int_D \inner{\alpha,\beta}e^{-\varphi} d\lambda,
	\|\alpha \|^2_{\varphi}=\iinner{\alpha,\alpha}_{\varphi},
\end{aligned}
$$
for simplicity, we write $\inner{\alpha, \beta}=\inner{\alpha, \beta}_{0}$.

Let $L^2_{(0,q)} (D,\varphi)$ be the completion of $\wedge^{0,q}_c(D)$ with respect to the inner product $\iinner{\cdot,\cdot}_{\varphi}$,
then $L^2_{(0,q)}(D,\varphi)$ is a Hilbert space and
$\dbar:L^2_{(0,q)}(D,\varphi)\ra L^2_{(0,q+1)}(D,\varphi)$ is a  closed and densely defined operator.
Let $\dbar_\varphi^{*}$ be the Hilbert adjoint of $\dbar$.
For convenience,
we denote $L^2 (D, \varphi):=L^2_{(0,0)} (D, \varphi)$.

Since
$$
\dbar \circ \dbar =0,
$$
we can define the weighted $\dbar$-complex as follows
$$
0 \ra L^2_{p,0}(D,\varphi) \xrightarrow[]{\dbar} L^2_{p,1}(D,\varphi) \xrightarrow[]{\dbar} L^2_{p,2}(D,\varphi) 
\xrightarrow[]{\dbar} \cdots \xrightarrow[]{\dbar} L^2_{p,n}(D,\varphi) \ra 0.
$$

One can show that for $\alpha\in \wedge^{0,q}(\bar{D})$ that
\begin{equation}\label{eq:domain of dbar^*}
\alpha\in \text{Dom}(\bar\partial^*_\varphi)\Longleftrightarrow  \alpha \in \Ker(
\sum^{n}_{i=1} \frac{\pa \rho}{\pa z_{i}} \frac{\pa }{\pa \bar{z}_{i}} \lrcorner ) \; on \; \pa D.
\end{equation}
Note that the condition on the right-hand side of the above formula is independent of the weight $\varphi$.

We recall the following Morrey-Kohn-H\"ormander identity for $\dbar$-operator with boundary term.
\begin{lem}\label{lem:Bochner-identity-boundary-term-dbar-operator}
For $\alpha=\sum_{|I|=q}^\prime\alpha_I d\bar z_I \in \wedge^{0,q}(\bar{D}) \cap \text{Dom}(\bp_\varphi^*) $, we have
\begin{equation}
\int _{ D} (|\bp^{*}_{\varphi } \alpha |^2+|\bp \alpha |^2)e^{-\varphi}d \lambda=\| \overline{\nabla} \alpha  \|^2_{\varphi}+
\iinner{\textup{Levi}_{\varphi} \alpha, \alpha}_{\varphi}
+\int_{\pa D} \inner{\Levi_{\rho} \alpha, \alpha}e^{-\varphi} dV
\end{equation}
where $d\lambda$ and $dV$ denote the Lebsgue measures on $\mc^n$ and $\partial D$, respectively,
$$\| \overline{\nabla} \alpha  \|^2_{\varphi}=\sum \int_D |\frac{\partial\alpha_I}{\partial\bar z_j}|^2_\varphi,$$   
$$\langle \Levi_\rho\alpha,\alpha\rangle=\sum_{i,j}\sum_{|K|=q-1}^\prime\rho_{i\bar j}\alpha_{iK}\bar{\alpha_{jK}}.$$
\end{lem}

\subsection{Notations for variants of partial convexity}

Next, we recall the concept of hyper-$q$-convex domains, which was introduced by Grauert and Riemenschneider (\cite{Grauert-Riemenschneider}) as a natural generalization of pseudoconvex domains.

\begin{defn}[\cite{Grauert-Riemenschneider},\cite{SYT-1982}]
 Let $D$ be a bounded domain in $\mc^n$ with smooth boundary defining functions $\rho$.
 The boundary \(\partial D\) of \(D \) is said to be \textit{hyper-\(q\)-convex} (respectively \textit{strongly hyper-\(q\)-convex}) at \( z_0 \in \partial D \) if the sum of any \( q \) eigenvalues of the 
 Levi form on the complex tangent space $T_{z_0}\partial D$ is nonnegative (respectively positive). 
 When \(\partial D\) is hyper-\(q\)-convex at every point of \(\partial D\), we simply say that \(\partial D\) is hyper-\(q\)-convex.
\end{defn}

The following theorem provides equivalent characterizations of hyper-$q$-convex domains.

\begin{thm}[\cite{Ho-1991}]\label{prop:properties-hyper-q-convex-domain}
Let $\Omega$ be a domain in $\mc^n$ with $\text{C}^2$-smooth boundary. The following conditions are equivalent:
\begin{itemize}
    \item[(i)] $\Omega$ is a hyper-$q$-convex domain.
    \item[(ii)] The sum of any $q$ eigenvalues of the Levi form on the complex tangent space $T^{\mc}\partial \Omega$ is non-negative.
    \item[(iii)] Let $\rho$ be a defining function for $\Omega$. For any smooth $(0,q)$-form $\alpha=\sum \alpha_{I}d\bar{z}_{I} \in \wedge^{0,q}(\bar\Omega)$ satisfying
        \begin{equation}\label{eq:case-dbar-neumann}
        \sum^{n}_{i=1} \alpha_{iK}\frac{\partial \rho}{\partial z_{i}}=0 \text{ for all } |K|=q-1 \text{ on } \partial\Omega,
        \end{equation}
        the following inequality holds:
        \begin{equation}\label{eq:defn-hyper-q-convex}
	        \langle \Levi_{\rho} \alpha, \alpha \rangle \geq 0 \text{ on } \partial \Omega.	
		\end{equation}
\end{itemize}
\end{thm}

The following lemma is also required in the proof.
\begin{lem}[\cite{Ho-1991}\cite{Kohn63}] \label{lemma:q-subharmonic-invariant}
	The condition \eqref{eq:defn-hyper-q-convex} is invariant under a unitary change of coordinates.
\end{lem}

\section{Characterize the domain with smooth partial pseudoconvex boundary}\label{sec:partial-pseducocovex}

We prove Theorem \ref{thm:main-2} in this section.

Let $D$ be a bounded domain in $\mc^n$ with smooth boundary.
For any boundary point \( p_0 \) of \( D \) and any \( 0 \leq i \leq n \), given any $(0,i)$-form \( \xi \in \wedge^{0,i}_{p_0}(\bar D) \) satisfying (\ref{eq:case-dbar-neumann}) at $p_0$, does there exist an \( \alpha \in \wedge^{0,i}(\bar{D}) \cap \text{Dom}(\bar{\partial}^*) \) such that $\alpha(p_0) = \xi$ and $\bar{\partial} \alpha = 0$?

It is clear that if \( \alpha_1 \in \wedge^{0,i}(\bar{D}) \cap \text{Dom}(\bar{\partial}^*) \cap \text{Ker}(\bar{\partial}) \) and \( \alpha_1(p_0) = \xi_1 \), and \( \alpha_2 \in \wedge^{0,j}(\bar{D}) \cap \text{Dom}(\bar{\partial}^*) \cap \text{Ker}(\bar{\partial}) \) with \( \alpha_2(p_0) = \xi_2 \), then \( \alpha_1 \wedge \alpha_2 \in \wedge^{0,i+j}(\bar{D}) \cap \text{Dom}(\bar{\partial}^*) \cap \text{Ker}(\bar{\partial})  \), and \( \alpha_1 \wedge \alpha_2 (p_0) = \xi_1 \wedge \xi_2\). 

The following lemma 
present the solution to the above problem in the case when \( i = 1 \).

\begin{lem}[\cite{Deng-Zhang}] \label{lem:exist-dbar-closed-form}
For any boundary point \( p_0 \) of \( D \), given any $(0,1)$-form \( \xi = \sum \xi_i d\bar{z}_i \in \wedge^{0,1}_{p_0}(\bar D) \) satisfying (\ref{eq:case-dbar-neumann}) at $p_0$, there exists a function \( h \) on \( \mathbb{C}^n \) such that
$$
\sum_{i=1}^{n} \frac{\partial \rho}{\partial z_i} \frac{\partial h}{\partial \bar{z}_i} = 0
$$
and
$$
\bar{\partial} h(p_0) = \sum \xi_i d\bar{z}_i.
$$
\end{lem}

Now we are ready to prove the main result of this section.

\begin{thm}(= Theorem \ref{thm:main-2})
Let $D$ be a bounded domain in $\mc^n$ with smooth boundary, $ 1\leq q \leq n-1$.
Suppose that for any strictly plurisubharmonic function $\varphi$ of the form: 
	$$
	\varphi=a\|z-z_0\|^2-b, \forall a, b > 0, \forall z_0 \in \mc^n,
	$$ 
	and any $\bar\partial$-closed form $f\in \wedge^{0,q}(\bar{D}) \cap \text{Dom}(\bar\partial^*_{\varphi})$,
	the equation $\bar\partial u=f$ is solvable on $D$ with the estimate
		$$
		\int_{D} |u|^2 e^{-\varphi} d\lambda    \leq \int_{D} \inner{\textup{Levi}_{\varphi}^{-1}f,f } e^{-\varphi }d\lambda    =\frac{1}{aq}\int_D|f|^2e^{-\varphi } d\lambda,
		$$
then $D$ is hyper-$q$-convex.
\end{thm}

\begin{proof}
	For any $\alpha \in \wedge^{0,q}(\bar{D})\cap \Ker (\dbar) \cap \Dom (\dbar^*)$, we have
	$$
	\begin{aligned}
	 \langle  \langle \alpha ,f \rangle   \rangle_{\varphi}^2 &= \langle  \langle \alpha ,\bar\partial u \rangle   \rangle_{\varphi}^2=\langle   \langle   \bar\partial_{\varphi}^{*}\alpha ,u \rangle  \rangle_{\varphi}^2 \\
	&\leq \langle   \langle  \bar\partial_{\varphi}^{*}\alpha ,\bar\partial_{\varphi}^{*}\alpha \rangle  \rangle_{\varphi}  \cdot  \langle   \langle   u ,u \rangle  \rangle_{\varphi}.
	\end{aligned}
	$$

	By the estimate of $\|u \|_{\varphi}^2$ and Lemma \ref{lem:Bochner-identity-boundary-term-dbar-operator}, we get
$$
\begin{aligned}
\langle   \langle  \alpha ,f \rangle  \rangle_{\varphi} ^2 \leq&  \frac{1}{aq}\|f\|^2_{\varphi}\cdot\\
&\left(\| \overline{\nabla} \alpha  \|^2_{\varphi}+
\iinner{\textup{Levi}_{\varphi} \alpha, \alpha}_{\varphi}
+\int_{\pa D} \inner{\Levi_{\rho} \alpha, \alpha}e^{-\varphi} dV  \right)
\end{aligned}
$$

We now relate $\alpha$ and $f$ by setting
$$
\alpha=\Levi^{-1}_{\varphi} f,
$$
this substitution yields:
\begin{equation}\label{eq:condition for (0,q)-form-local in proof}
	\| \overline{\nabla} \alpha  \|^2_{\varphi}
+\int_{\pa D} \inner{\text{Levi}_{\rho} \alpha, \alpha}e^{-\varphi} dV\geq 0.
\end{equation}

Next, we argue by contradiction. 

If $D$ has at least one non hyper-$q$-convex boundary point $p_0 \in \partial D$,
by definition, there exists a $(0,q)$-form $\beta \in \wedge^{0,q}_{p_0}(\bar D) $  satisfying (\ref{eq:case-dbar-neumann}) at $p_0$  and a constant $c>0$ such that
$$
\inner{\text{Levi}_{\rho} \beta, \beta} = -2c < 0.
$$
Through a unitary coordinate transformation, we take the coordinates $\{z_1, \dots, z_n\}$ such that:
\begin{itemize}
    \item[(1)] $p_0$ is the origin $O = (0, \dots, 0)$.
    \item[(2)] $\displaystyle -\frac{\pa}{\pa z_n}$ is the inward-pointing normal vector at $p_0$.
    \item[(3)] $\displaystyle \left\{ \frac{\pa}{\pa z_1}, \dots, \frac{\pa}{\pa z_{n-1}} \right\}$ are the eigenvectors corresponding to the eigenvalues
    $$
    k_1 \leq \dots \leq k_{n-1}
    $$
    of the Levi form of $\partial D$ restricted to the holomorphic tangent bundle $T^{1,0}_{p_0}(\partial D)$ at $p_0$.
\end{itemize}
In this coordinate, the $(0,q)$-form $\beta = d\bar{z}_1 \wedge \cdots \wedge d\bar{z}_q \in \wedge^{0,q}_{p_0}(\bar D) $  satisfying (\ref{eq:case-dbar-neumann}) at $p_0$
and
$$
\inner{\text{Levi}_{\rho} \beta, \beta}= \sum_{i=1}^{q}(k_i)\|\xi\|^2 = -2c < 0 \text{ for some } \xi.
$$

According to Lemma \ref{lem:exist-dbar-closed-form}, for any $1 \leq i \leq q$, take $\alpha_i \in \wedge^{0,1}(\bar{D}) \cap \text{Dom}(\bar{\partial}^*) \cap \text{Ker}(\bar{\partial})$ such that $\alpha_i(p_0) = d\bar{z}_i$. Since
$$
T_{\rho}(\alpha_1 \wedge \cdots \wedge \alpha_q) = \sum_{i=1}^{q}( \alpha_1 \wedge \cdots \wedge T_{\rho}\alpha_i \wedge \cdots \wedge \alpha_q) = 0,
$$
where 
$$
T_{\rho}=\sum^{n}_{i=1} \frac{\pa \rho}{\pa z_{i}} \frac{\pa }{\pa \bar{z}_{i}} \lrcorner: \wedge^{0,i}(\bar{D}) \ra \wedge^{0,i-1}(\bar{D}), \;\;\; \forall 1 \leq i \leq n, 
$$
there exists a closed $(0,q)$-form $\alpha = \alpha_1 \wedge \cdots \wedge \alpha_q \in \wedge^{0,q}(\bar{D}) \cap \text{Dom}(\bar{\partial}^*) \cap \text{Ker}(\bar{\partial})$ and a constant $0 < r_0 \ll 1$ such that
$$
\inner{\text{Levi}_{\rho} \alpha, \alpha} (z) < -c
$$
for any $z \in B_{r_0} \cap \partial D$, where $B_{r_0} = \{z \in \mathbb{C}^n \mid \|z - p_0\| < r_0\}$.

For any $s > 0$, let $\varphi_s(x) = s(\|z\|^2 - r_0^2)$,
replacing $\varphi_s$ in \eqref{eq:condition for (0,q)-form-local in proof}, 
we have
\begin{equation}\label{eq:condition for (0,q)-form scaling-local in proof}
\int_{D} \|\bar{\nabla} \alpha \|^2 e^{-\varphi_s} d\lambda
 + \int_{\partial D} \inner{\text{Levi}_{\rho} \alpha, \alpha} e^{-\varphi_s} d\lambda \geq 0,
\end{equation}
which holds for any $s > 0$.

Notice that when $z \in \mathbb{C}^n \setminus \overline{B_{r_0}}$ and $s \to +\infty$, we have $e^{-\varphi_s(z)} \to 0$. At the same time, there exists a constant $M > 0$ such that
$$
\|\bar{\nabla} \alpha(z) \|^2 \leq M
$$
for any $z \in \bar{\Omega}$ since $D$ is bounded. 
Therefore, we obtain:
\begin{equation}\label{eq:key inequality-local in proof}
\liminf_{s \to +\infty} \left( M \int_{D \cap B_{r_0}} e^{-\varphi_s} d\lambda - c \int_{\partial D \cap B_{r_0}} e^{-\varphi_s}dV \right) \geq 0.
\end{equation}

Noticing that the real dimension of $D \cap B_{r_0}$ is $2n$ and that the real dimension of $\partial D \cap B_{r_0}$ is $2n-1$.

Next, we estimate each term in \eqref{eq:key inequality-local in proof}. For the first term, it is easy to see that there exists a constant \( M' > 0 \) such that:
$$
M \int_{B^{2n}_r \cap D} e^{-s \|x\|^2}d\lambda = M' \int_0^{r_0} r^{2n-1} e^{-s(r^2 - r_0^2)} dr.
$$
For the second term, we need a more detailed discussion.

By assumption, there exists a smooth function \( g(x_1, \cdots, x_{2n-1}) \) defined on \( \mathbb{R}^{2n-1} \) such that, in a neighborhood of \( \partial D \), the boundary is represented by
$$
g(x_1, \cdots, x_{2n-1}) - x_n = 0,
$$
with \( g(0, \cdots, 0) = 0 \).

Since \( \nabla \rho(O) = -\frac{\partial}{\partial x_{2n}} \), we have \( \frac{\partial g}{\partial x_i}(0, \cdots, 0) = 0 \) for \( i = 1, \cdots, 2n-1  \). 

In a neighborhood of \( p_0 \), we choose local coordinates \( (x_1, \cdots, x_{2n-1}) \) on \( \partial D \), then near \( x_0 \), the function
$$
\|x\|^2 \bigg|_{\partial D} = x_1^2 + \cdots + x_{2n-1}^2 + g(x_1, \cdots, x_{2n-1})^2
$$
has a non-degenerate critical point of index 0. By Morse's lemma, in a neighborhood of \( x_0 \), there exist local coordinates \( (y_1, \cdots, y_{2n-1}) \) such that
$$
\|x\|^2 \bigg|_{\partial D} = y_1^2 + \cdots + y_{2n-1}^2.
$$
Therefore, locally,
$$
\begin{aligned}
\sigma : B^{2n-1}_{r_0} &\to B^{2n}_{r_0} \cap \partial D \\
(x_1, \cdots, x_{2n-1}) &\mapsto (y_1, \cdots, y_{2n-1})
\end{aligned}
$$
is a homeomorphism. 
Hence, there exists a constant \( A > 0 \) such that:
\begin{equation*}
\begin{split}
\int_{B^{2n}_{r_0} \cap \partial D} e^{-s \|x\|^2}d\lambda &= A \int_{B^{2n-1}_{r_0}} e^{-s(x_1^2 + \cdots + x_{2n-1}^2 - r_0^2)} d\mu \\
&= A' \int_0^{r_0} r^{2n-2} e^{-s(r^2 - r_0^2)} dr \\
&\geq A' r_0^{-1} \int_0^{r_0} r^{2n-1} e^{-s(r^2 - r_0^2)} dr.
\end{split}
\end{equation*}
Choosing \( r_0 \) small enough such that \( cA'r_0^{-1} > M' + 1 \), we then have
$$
M \int_{D \cap B_{r_0}} e^{-\varphi_s}d\lambda
- c \int_{\partial D \cap B_{r_0}} e^{-\varphi_s} dV\leq -\int_0^{r_0} r^{2n-1} e^{-s(r^2 - r_0^2)} dr.
$$
As \( s \to +\infty \), this asymptotic behavior implies that the right-hand side of the above tends to \( -\infty \), contradicting the non-negativity condition in \eqref{eq:key inequality-local in proof}, this contradiction completes the proof.
\end{proof}

\section{The Union Problem for the hyper-$q$-convex domain}\label{sec:proof-union}

We prove Theorem \ref{thm:main-1} in this section.
Let us recall the notable H\"ormander $L^2$ existence theorem for the hyper-$q$-convex domain. 
\begin{thm}[\cite{Hor65},\cite{Ji-Tan-Yu-2015}]
	\label{thml2existence}
  Set $1 \leq q \leq n-1$, let $D$ be a hyper-$q$-convex domain in $\mc^n$. Assume that $\varphi$ is a smooth strictly plurisubharmonic function on $\mc^n$,  then for any $f\in L^2_{(0,q)}(D,\varphi)\cap\Ker (\dbar)$ satisfying with
  \begin{align*}
              \int_D \inner{\Levi_{\varphi}^{-1}f,f}_{\varphi}  d\lambda   < +\infty,
  \end{align*} there is a
  $u\in L^2_{(0,q-1)}(D,\varphi)$ such that $\dbar u=f$
  and
   $$
   \int_D|u|^2_{\varphi}  d\lambda
   \leq
   \int_D \inner{\Levi_{\varphi}^{-1}f,f}_{\varphi}  d\lambda.
   $$
\end{thm}

Now we are ready to prove Theorem \ref{thm:main-1}.

\begin{thm}[= Theorem \ref{thm:main-1}]
	Let $D$ be a bounded domain in $\mc^n$ with smooth boundary. 
	Let $\{D_j\}$ be a sequence of open subsets of $D$ with $D_j\subset D_{j+1} $ and $\bigcup_jD_j=D$.  Assume that all $D_j$ are hyper-$q$-convex, 
then $D$ is hyper-$q$-convex.	
\end{thm}

\begin{proof}
	Let $\varphi$ be a smooth strictly plurisubharmonic function on $\mc^n$ and $f \in \wedge^{0,q}(\bar{D}) \cap \text{Dom}(\bar\partial^*_{\varphi})\cap \Ker(\dbar)$ satisfies
   $$
   \int_{D} \langle \Levi_\varphi^{-1} f, f\rangle e^{-\varphi} d\lambda<+\infty
   $$
  Then by Theorem  \ref{thml2existence}, there exists $u_j\in L^2_{(0,q-1)}(D_j,\varphi)$ such that $\dbar u_j=f$ on $D_j$ with the estimate
   $$
   \int_{D_j}|u_j|^2 e^{-\varphi} d\lambda
   \leq
   \int_{D_j} \langle \Levi_\varphi^{-1} f, f\rangle e^{-\varphi} d\lambda\leq
   \int_D \langle \Levi_\varphi^{-1} f, f\rangle e^{-\varphi} d\lambda.
   $$
 This means $\{u_j\}$ is a bounded subset in the Hilbert space $L^2_{(0,q-1)}(D,\varphi )$.
   Hence there is a subsequence $\{u_j\}$, assume to be $\{u_j\}$ itself without loss of generality, that  converges weakly in $L^2_{(0,q-1)}(D,\varphi)$ to some $u$. Note that we also have $\dbar u=f$
   in the sense of distribution. And we have 
   \begin{equation*}
	   \int_{D}^{} |u|^2 e^{-\varphi} d\lambda
	   \le \limsup_{j\to +\infty}\int_{D}^{} |u_{j}|^2 e^{-\varphi} d\lambda 
	   \leq 
	   \int_D \langle \Levi_\varphi^{-1} f, f\rangle e^{-\varphi} d\lambda.
	   \end{equation*}
   Thus, $$
   \int_D|u|^2 e^{-\varphi} d\lambda
   \leq
   \int_D \langle \Levi_\varphi ^{-1} f, f\rangle e^{-\varphi} d\lambda.
   $$
Then by Theorem \ref{thm:main-2},
$D$ is hyper-$q$-convex.
\end{proof}

\section{Convex Analogy}\label{sec:convex-analogy}

\subsection{The weighted $d$-complex}

The coordinate on $\mr^n$ will be denoted by $x=(x_1,\cdots, x_n)$, $(j=1,\cdots, n).$
We assume that $D$ is a bounded domain in $\mr^n$ with smooth boundary.
Then there exits a smooth function $\rho:\mr^n \ra \mr$  such that
$$
D=\left\{x \in \mr^n; \rho(x)<0 \right\}
$$
and $\nabla \rho|_{\pa D} \neq 0$, where
$$\nabla\rho=\sum_{j=1}^n\frac{\partial\rho}{\partial x_j}\frac{\partial}{\partial x_j}$$
is the gradient of $\rho$.
We take a normalization such that $|\nabla\rho|\equiv 1$ on $\partial D$.

Let $D$ be a domain in $\mr^n$.
We denote by $\wedge^{q}(D)$ the space of smooth  $q$-forms on $D$, for any $0\leq q\leq n$ (0-forms are just smooth functions), and
$\wedge^{q}_c(D)$ the elements in
$\wedge^{q}(D)$ with compact support.
Let $\varphi$ be a real-valued continuous function on $D$.
Given $\alpha=\sum_I\alpha_Id  x_I, \beta=\sum_I\beta_Idx_I\in \wedge^{q}(D)$, we define the products of $\alpha$ and $\beta$ and the corresponding norm with respect to $\varphi$ as follows:
$$
\begin{aligned}
	\inner{\alpha, \beta}_{\varphi} &= \sum_I \alpha_I \cdot \beta _I e^{-\varphi},
	|\alpha|^2_{\varphi}=\inner{\alpha,\alpha}_{\varphi}. \\
	\iinner{\alpha, \beta}_{\varphi} &=\int_D \inner{\alpha,\beta}e^{-\varphi} d\lambda,
	\|\alpha \|^2_{\varphi}=\iinner{\alpha,\alpha}_{\varphi}.
\end{aligned}
$$
Let $L^2_{q} (D,\varphi)$ be the completion of $\wedge^{q}_c(D)$ with respect to the inner product $\iinner{\cdot,\cdot}_{\varphi}$,
then $L^2_{q}(D,\varphi)$ is a Hilbert space and
$d:L^2_{q}(D,\varphi)\ra L^2_{q+1}(D,\varphi)$ is a  closed and densely defined operator.
Let $d_\varphi^{*}$ be the Hilbert adjoint of $d$.
For convenience,
we denote $L^2 (D, \varphi):=L^2_{0} (D, \varphi)$.

Since
$$
d \circ d =0,
$$
we can also define the weighted $d$-complex as follows
$$
0 \ra L^2_{0}(D,\varphi) \xrightarrow[]{d} L^2_{1}(D,\varphi) \xrightarrow[]{d} L^2_{2}(D,\varphi) 
\xrightarrow[]{d} \cdots \xrightarrow[]{d} L^2_{n}(D,\varphi) \ra 0.
$$

One can show that for $\alpha\in \wedge^{q}(\bar{D})$ that
\begin{equation}\label{eq:domain of d^*}
\alpha\in \text{Dom}(d^*_\varphi)\Longleftrightarrow  \alpha \in \Ker(\sum\frac{\partial\rho}{\partial x_j}\frac{\partial}{\partial x_j} \lrcorner ) \; on \; \pa D.
\end{equation}
Note that the condition on the right-hand side of the above formula is independent of the weight $\varphi$.

We recall the following Morrey-Kohn-H\"ormander identity for $d$-operator with boundary term.
\begin{lem}\label{lem:Bochner-identity-boundary-term-d-operator}
For $\alpha=\sum_{|I|=q}^\prime\alpha_I dx_I \in \wedge^{q}(\bar{D}) \cap \Dom(d_\varphi^*) $, we have
\begin{equation}
\int_{D} (|d^{*}_{\varphi } \alpha |^2+|\bp \alpha |^2)e^{-\varphi}d \lambda=\| \nabla \alpha  \|^2_{\varphi}+
\iinner{\Hess_{\varphi} \alpha, \alpha}_{\varphi}
+\int_{\pa D} \inner{\Hess_{\rho} \alpha, \alpha}e^{-\varphi} d\lambda
\end{equation}
where $d\lambda$ denotes the Lebsgue measures on $\mr^n$, and $$\|  {\nabla} \alpha  \|^2_{\varphi}=\sum \int_D |\frac{\partial\alpha_I}{\partial  x_j}|^2_\varphi,$$
$$\langle \Hess_\rho\alpha,\alpha\rangle=\sum_{i,j}\sum_{|K|=q-1}^{\prime}\rho_{i j}\alpha_{iK} {\alpha_{jK}}.$$
\end{lem}

\subsection{Proof of Theorem \ref{thm:main-3}}

First, we consider the following construction problem. Note that $\xi=\sum_I^{\prime}\xi_Idx_I\in\wedge^{i}_{p_0}(\partial D)$ means $\xi\in\wedge^{i}_{p_0}(\bar D)$ satisfying
$$\sum \frac{\partial\rho}{\partial x_j}\xi_{jK}=0 \text{ for all } |K|=q-1 \text{ at } p_0.$$

For any boundary point \( p_0 \) of \( D \), and for any \( 0 \leq i \leq n \), given any \( i \)-form \( \xi \in \wedge^{i}_{p_0}(\partial D) \)  at \( p_0 \), is there a form \( \alpha \in \wedge^{i}(\bar{D}) \cap \text{Dom}(d^*) \) such that $\alpha(z) = \xi $ and $d \alpha = 0$ ?

It's clear that, if
\[
\alpha_1 \in \wedge^i(\bar{D}) \cap \text{Dom}(d^*) \cap \Ker(d), \quad \alpha_1(p_0) = \xi_1
\]
and
\[
\alpha_2 \in \wedge^j(\bar{D}) \cap \text{Dom}(d^*) \cap \Ker(d), \quad \alpha_2(p_0) = \xi_2,
\]
then
\[
\alpha_1 \wedge \alpha_2 \in \wedge^{i+j}(\bar{D}) \cap \text{Dom}(d^*) \cap \Ker(d),
\]
and
\[
\alpha_1 \wedge \alpha_2(p_0) = \xi_1 \wedge \xi_2 \in \wedge^{i+j}_{p_0}(\partial D).
\]

Next, we show that the problem has a solution for \( i = 1 \).

\begin{lem}[\cite{Deng-Zhang}]\label{lem:exist-d-closed-form}
For any boundary point \( p_0 \) of \( D \), given any 1-form
\[
\xi = \sum \xi_i dx_i \in \wedge^1_{p_0}(\partial D)
\]
at \( p_0 \), there exists a function \( u \) on \( \mathbb{R}^n \) such that
$$
\sum_{i=1}^n \frac{\partial \rho}{\partial x_i} \frac{\partial u}{\partial x_i} = 0
$$
and
$$
du(p_0) = \sum \xi_i dx_i.
$$
\end{lem}

\begin{proof}
Let \( u_1 \in C^\infty(\partial D) \) be a smooth function on \( \partial D \) such that \( \nabla u_1(p_0) = (\xi_1, \dots, \xi_n) \). We extend \( u_1 \) to a smooth function \( u_2 \) in a neighborhood \( U \) of \( \partial D \), constant in the normal direction, with \( U \) containing \( \bar{D} \).

Next, choose a cutoff function \( \chi \in C^\infty(\mathbb{R}^n) \) that is identically 1 in a small neighborhood \( V \) of \( \partial D \), and 0 outside \( V \). Let \( u = \chi u_2 \). Note that since \( u_2 \) is constant in the normal direction, we have
$$
\nabla \rho \cdot \nabla u = 0,
$$
which implies
$$
\nabla \rho \lrcorner du = 0.
$$
Moreover, we have \( du(p_0) = \xi \), completing the proof.
\end{proof}

Now we are ready to prove the main results of this section.

\begin{thm}(= Theorem \ref{thm:main-3})
	Let $D$ be a bounded domain in $\mr^n$ with smooth boundary, $ 1\leq q \leq n-1$.
Suppose that for any strictly convex function $\varphi$ of the form: 
	$$
	\varphi=a\|x-x_0\|^2-b, \forall a, b > 0, \forall x_0 \in \mr^n,
	$$ 
	and any $d$-closed form $f\in \wedge^{q}(\bar{D}) \cap \text{Dom}(d^*_{\varphi})$,
	the equation $d u=f$ is solvable on $D$ with the estimate
		$$
		\int_{D} |u|^2 e^{-\varphi} d\lambda    \leq \int_{D} \inner{\Hess_{\varphi}^{-1}f,f } e^{-\varphi }  d\lambda  =\frac{1}{2aq}\int_D|f|^2e^{-\varphi } d\lambda ,
		$$
then $D$ is $q$-convex.
\end{thm}

\begin{proof}

	For any $\alpha \in \wedge^{q}(\bar{D})\cap \Ker (d) \cap \Dom (d^*)$, we have
	$$
	\begin{aligned}
	 \langle  \langle \alpha ,f \rangle   \rangle_{\varphi}^2 &= \langle  \langle \alpha ,d u \rangle   \rangle_{\varphi}^2=\langle   \langle   d_{\varphi}^{*}\alpha ,u \rangle  \rangle_{\varphi}^2 \\
	&\leq \langle   \langle  d_{\varphi}^{*}\alpha ,d_{\varphi}^{*}\alpha \rangle  \rangle_{\varphi}  \cdot  \langle   \langle   u ,u \rangle  \rangle_{\varphi}.
	\end{aligned}
	$$

	By the estimate of $\|u \|_{\varphi}^2$ and Lemma \ref{lem:Bochner-identity-boundary-term-d-operator}, we get
$$
\begin{aligned}
\langle   \langle  \alpha ,f \rangle  \rangle_{\varphi} ^2 \leq&  \frac{1}{2aq}\|f\|^2_{\varphi}\cdot\\
&\left(\| \overline{\nabla} \alpha  \|^2_{\varphi}+
\iinner{\Hess_{\varphi} \alpha, \alpha}_{\varphi}
+\int_{\pa D} \inner{\Hess_{\rho} \alpha, \alpha}e^{-\varphi} dV  \right)
\end{aligned}
$$

We now relate $\alpha$ and $f$ be setting
$$
\alpha=\Hess^{-1}_{\varphi} f
$$
then we get
\begin{equation}\label{eq:condition for q-form-local in proof-2}
	\| \nabla \alpha  \|^2_{\varphi}+
\int_{\pa D} \inner{\text{Hess}_{\rho} \alpha, \alpha}e^{-\varphi} dV \geq0.
\end{equation}

Suppose \( D \) has at least one non-\( q \)-convex boundary point \( p_0 \in \partial D \), 
by definition,
there exists a \( q \)-form \( \xi \in \wedge^{q}T^*_{p_0}(\partial D) \) at \( p_0 \) and a constant $c>0$, such that
$$
\langle\text{Hess}_{\rho}\xi, \xi\rangle <-2c<0.
$$
By an affine coordinate transformation, choose the coordinates \( \left\{ x_1, \dots, x_n \right\} \) such that:

\begin{enumerate}
	\item \( p_0 \) is the origin \( O = (0, \dots, 0) \).
	\item \( -\frac{\partial}{\partial x_n} \) is the inward normal vector at \( p_0 \).
	\item \( \left\{ \frac{\partial}{\partial x_1}, \dots, \frac{\partial}{\partial x_{n-1}} \right\} \) are the eigenvectors of the second fundamental form of \( \partial D \) restricted to \( T_{p_0}(\partial D) \), with corresponding eigenvalues
    $$
    t_1 \leq \dots \leq t_{n-1}.
    $$
\end{enumerate}

In this coordinate, 
the \( q \)-form \( \xi = dx_1 \wedge \dots \wedge dx_q \in \wedge^q_{p_0}(\partial D) \) satisfying
$$
\langle\text{Hess}_{\rho}\xi ,\xi\rangle = \sum_{i=1}^q t_i \|\xi\|^2 <-2c<0 \text{ for some } \xi.
$$
By Lemma \ref{lem:exist-d-closed-form}, for any \( 1 \leq i \leq q \), we take \( \alpha_i \in \wedge^1(\bar{D}) \cap \text{Dom}(d^*) \cap \Ker(d) \) such that \( \alpha_i(p_0) = dx_i \). Since
$$
\nabla \rho \lrcorner(\alpha_1 \wedge \cdots \wedge \alpha_q) = \sum_{i=1}^q (\alpha_1 \wedge \cdots \wedge \nabla \rho \lrcorner\alpha_i \wedge \cdots \wedge \alpha_q) = 0,
$$
there exists a closed \( q \)-form \( \alpha = \alpha_1 \wedge \cdots \wedge \alpha_q \in \wedge^q(\bar{D}) \cap \text{Dom}(d^*) \cap \Ker(d) \) and a constant \( r_0 > 0 \), such that
$$
\langle \text{Hess}_{\rho} \alpha, \alpha \rangle(x) <-c< 0
$$
for all \( x \in B_{r_0} \cap \partial D \), where \( B_{r_0} = \{ x \in \mathbb{R}^n \mid \|x - x_0\| < r_0 \} \).

For any \( s > 0 \), define
$$
\psi_s(x) = s(\|x\|^2 - r_0^2).
$$
Substituting \( \psi_s \) into \eqref{eq:condition for q-form-local in proof-2} in place of \( \psi \), we have
\begin{equation}\label{eq:condition for q-form scaling-local in proof}
\int_D \|\nabla \alpha\|^2 e^{-\psi_s} d\lambda + \int_{\partial D} \langle \text{Hess}_{\rho} \alpha, \alpha \rangle e^{-\psi_s} dV\geq 0,
\end{equation}
for all \( s > 0 \).

Notice that for \( x \in \mathbb{R}^n \setminus \overline{B_{r_0}} \), as \( s \to +\infty \), we have
$$
e^{-\psi_s(x)} \to 0,
$$
so
$$
\liminf_{s \to +\infty} \left( \int_{D \cap B_{r_0}} \|\nabla \alpha\|^2 e^{-\psi_s} d\lambda + \int_{\partial D \cap B_{r_0}} \langle \text{Hess}_{\rho} \alpha, \alpha \rangle e^{-\psi_s} dV \right) \geq 0.
$$
Since \( D \) is bounded, there exists a constant \( M > 0 \) such that
$$
\|\nabla \alpha(x)\|^2 \leq M
$$
for all \( x \in \bar{D} \). Therefore, we obtain
\begin{equation}\label{eq:key inequality-local in proof-2}
\liminf_{s \to +\infty} \left( M \int_{D \cap B_{r_0}} e^{-\psi_s} d\lambda - c \int_{\partial D \cap B_{r_0}} e^{-\psi_s} dV \right) \geq 0,
\end{equation}
by the same reasoning as in \eqref{eq:key inequality-local in proof}, as $s \to +\infty$, the left-hand side of the above tends to $-\infty$, leading to a contradiction.
\end{proof}

\end{document}